\documentclass[11pt]{amsart}
\usepackage{latexsym, amssymb,hyperref}
\usepackage[T1]{fontenc}

\makeatletter
\@namedef{subjclassname@2010}{%
  \textup{2010} Mathematics Subject Classification}
\makeatother

\newtheorem{thm}{Theorem}[section]
\newtheorem{lemma}[thm]{Lemma}

\newtheorem{question}[thm]{Question}
\newtheorem{fact}[thm]{Fact}

\newtheorem*{conj*}{Conjecture}

\theoremstyle{definition}
\newtheorem{df}[thm]{Definition}

\theoremstyle{remark}

\renewcommand{\r}{\mathbb{R}}
\newcommand{\Z}{\mathbb{Z}}

\newcommand{\n}{\mathbb{N}}

\renewcommand{\to}{\rightarrow}

\def \<{\langle}
\def \>{\rangle}
\def \z {{\mathbb Z}}
\def \*Z {{{^*}\Z}}
\def \((  {(\!(}
\def \)) {)\!)}

\def \mo   {\text{mod}}

\def \st {\operatorname {st}}

\numberwithin{equation}{section}

\makeatletter

\def\indsym#1#2{%
  \setbox0=\hbox{$\m@th#1x$}%
  \kern\wd0%
  \hbox to 0pt{\hss$\m@th#1\mid$\hbox to 0pt{$\m@th#1^{#2}$}\hss}%
  \lower.9\ht0\hbox to 0pt{\hss$\m@th#1\smile$\hss}%
  \kern\wd0}

\def\nindsym#1#2{%
  \setbox0=\hbox{$\m@th#1x$}%
  \kern\wd0%
  \hbox to 0pt{\hss$\m@th#1\not$\kern1.4\wd0\hss}
  \hbox to 0pt{\hss$\m@th#1\mid$\hbox to 0pt{$\m@th#1^{\,#2}$}\hss}%
  \lower.9\ht0\hbox to 0pt{\hss$\m@th#1\smile$\hss}%
  \kern\wd0}

\allowdisplaybreaks[2]

\def\BD{\operatorname{BD}}

\begin{document}
\title{On a sumset conjecture of Erd\H{o}s}
\author[Di Nasso et. al.]{Mauro Di Nasso, Isaac Goldbring, Renling Jin,
Steven Leth, Martino Lupini, Karl Mahlburg}
\thanks{The authors were supported in part by the American Institute of Mathematics through its SQuaREs program.  I. Goldbring was partially supported by NSF grant DMS-1262210.  M. Lupini was supported by the York University Elia Scholars Program.  K. Mahlburg was supported by NSF Grant DMS-1201435.}
\address{Dipartimento di Matematica, Universita' di Pisa, Largo Bruno
Pontecorvo 5, Pisa 56127, Italy}
\email{dinasso@dm.unipi.it}
\address{Department of Mathematics, Statistics, and Computer Science,
University of Illinois at Chicago, Science and Engineering Offices M/C 249,
851 S. Morgan St., Chicago, IL, 60607-7045}
\email{isaac@math.uic.edu}
\address{Department of Mathematics, College of Charleston, Charleston, SC,
29424}
\email{JinR@cofc.edu}
\address{School of Mathematical Sciences, University of Northern Colorado,
Campus Box 122, 510 20th Street, Greeley, CO 80639}
\email{Steven.Leth@unco.edu}
\address{Department of Mathematics and Statistics, York University, N520
Ross, 4700 Keele Street, M3J 1P3, Toronto, ON, Canada}
\email{mlupini@mathstat.yorku.ca}
\address{Department of Mathematics, Louisiana State University, 228 Lockett
Hall, Baton Rouge, LA 70803}
\email{mahlburg@math.lsu.edu}

\begin{abstract}
Erd\H{o}s conjectured that for any set $A\subseteq \mathbb{N}$ with positive
lower asymptotic density, there are infinite sets $B,C\subseteq \mathbb{N}$
such that $B+C\subseteq A$. We verify Erd\H{o}s' conjecture in the case that $A$ has \emph{Banach} density exceeding $\frac{1}{2}$.  As a consequence, we prove that, for $A\subseteq \mathbb{N}$ with
positive Banach density (a much weaker assumption than positive lower density), we can find infinite $B,C\subseteq \mathbb{N}$ such
that $B+C$ is contained in the union of $A$ and a translate of $A$.  Both of the aforementioned
results are generalized to arbitrary countable
amenable groups. We also provide a positive solution to Erd\H{o}s'
conjecture for subsets of the natural numbers that are \emph{pseudorandom}.
\end{abstract}


\maketitle

\section{Introduction}

For $A\subseteq \mathbb{N}$, the \emph{lower (asymptotic) density
of $A$} is defined to be%
\begin{equation*}
\underline{d}(A):=\liminf_{n\rightarrow \infty }\frac{|A\cap \lbrack 1,n]|}{n%
}\text{.}
\end{equation*}%
Here, and throughout this paper, for $a,b\in \mathbb{N}$, $[a,b]$ denotes 
\begin{equation*}
\{c\in \mathbb{N}:\ a\leq c\leq b\}.
\end{equation*}%
Moreover if $A$ and $B$ are subsets of $\mathbb{N}$, then $A+B$ denotes the
sumset $\{a+b:\ a\in A\text{ and }b\in B\}$. In \cite{Erdos-problems} and 
\cite{Erdos-problemsIII} Erd\H{o}s conjectured the following generalization
of Hindman's theorem on sumsets (see \cite{Hindman}): If $A$ is a set of
natural numbers of positive lower density, then there is an infinite subset $%
A^{\prime }$ of $A$ such that $A^{\prime }+A^{\prime }$ is contained in a
translate of $A$. This density version of Hindman's theorem was inspired by
the celebrated Szemer\'{e}di theorem on arithmetic progressions (see \cite%
{Szemeredi}), which can be regarded as a density version of van der
Waerden's theorem from \cite{Waerden}. Later, Straus provided a
counterexample to this conjecture of Erd\H{o}s, as reported in \cite%
{Erdos-survey} on page 105. The conjecture was thus modified (cf.\ \cite%
{Nathanson} and page 85 of \cite{Erdos-Graham}) as follows.

\begin{conj*}[Erd\H{o}s]
If $A\subseteq \mathbb{N}$ has $\underline{d}(A)>0$, then there are two
infinite sets $B,C\subset \mathbb{N}$ such that $B+C\subset A$.
\end{conj*}

We will refer to this as \textquotedblleft Erd\H{o}s' $B+C$
conjecture\textquotedblright . Partial results on this conjecture have been
obtained by Nathanson in \cite{Nathanson}, where he proved in particular
that one can find an infinite set $B$ and an arbitrarily large finite set $F$
such that $B+F\subset A$. 


In this paper, we make progress on the $B+C$ conjecture by proving the
following ``one-shift'' version for sets of positive \emph{Banach density}, where, for $A\subseteq \mathbb{N}$, the (upper) Banach density of $A$ is defined to be
\begin{equation*}
\BD(A):=\lim_{n\rightarrow \infty }\sup_{m\in \mathbb{N}}\frac{|A\cap \lbrack m,m+n]|%
}{n}.
\end{equation*}%

\begin{thm}
\label{generalresult} If $\BD(A)>0$, then there are infinite $%
B,C\subseteq \mathbb{N}$ and $k\in \mathbb{N}$ such that $B+C\subseteq A\cup
(A+k)$.
\end{thm}

Observe that%
\begin{equation*}
\underline{d}\left( A\right) \leq \BD\left( A\right),
\end{equation*}%
whence the hypothesis of positive Banach density is weaker than the hypothesis of positive lower
density.

We also settle Erd\H{o}s' conjecture for sets of large Banach density.

\begin{thm}
\label{stronger} If $\BD(A)>\frac{1}{2}$, then there are infinite $%
B,C\subseteq \mathbb{N}$ such that $B+C\subseteq A$.
\end{thm}

We derive Theorem \ref{generalresult} from Theorem \ref{stronger} by showing
that every subset of the natural numbers of positive Banach density has
finitely many translates whose union has Banach density at least $\frac{1}{2}
$ and then use Ramsey's theorem to obtain our shifts.

In the proof of Theorem \ref{generalresult}, we will see that whether $b_i+c_j$ is in $A$ or $A+k$ depends only on whether or not $i<j $ holds, where $B=(b_i)$ and $C=(c_j)$ are increasing enumerations of $B$ and $C$ respectively.

We generalize both of the aforementioned results to the case of arbitrary countable amenable groups.  However, we present proofs for the two contexts separately as the proofs for subsets of the natural numbers are easier and/or require less technical machinery.

In the final section, we prove the $B+C$ conjecture for sets $A$ that are 
\emph{pseudorandom} in a precise technical sense. Here we remain in the
setting of sets of natural numbers as we do not know how to generalize one
of the key ingredients (Fact \ref{mauro}) to the setting of amenable groups.



We use \emph{nonstandard analysis} to derive our results and we assume that
the reader is familiar with elementary nonstandard analysis. For those not
familiar with the subject, the survey article \cite{Jin} contains a light
introduction to nonstandard methods with combinatorial number theoretic aims
in mind. The specific technical results from nonstandard analysis that we
will need are found in Section \ref{prelim}, where we review the Loeb
measure. In Section \ref{prelim}, we also recall the basic facts from the
theory of amenable groups that we need. In Sections \ref{high} and \ref%
{general}, we prove Theorems \ref{stronger} and \ref{generalresult}
respectively (as well as their amenable counterparts). In Section \ref{pseudo}, we prove Erd\H{o}s' conjecture for
pseudorandom sets. 

Throughout the paper, we do not include $0$ in the set $%
\mathbb{N}
$ of natural numbers.  Also, if $B$%
, $C$ are subsets of a group $G$, then $BC$ denote the set of products%
\begin{equation*}
\left\{ bc:\ b\in B\text{ and }c\in C\right\} \text{.}
\end{equation*}

\subsection{Acknowledgements}

This work was partly completed during a week long meeting at the American
Institute for Mathematics on June 3-7, 2013 as part of the SQuaRE
(Structured Quartet Research Ensemble) project ``Nonstandard Methods in
Number Theory.'' The authors would like to thank the Institute for the
opportunity and for the Institute's hospitality during their stay.

\section{Preliminaries}

\label{prelim}

\subsection{Loeb measure}

Throughout this paper, we always work in a countably saturated nonstandard
universe.

We recall the definition of Loeb measure, which is defined relative to a
fixed hyperfinite set $X$. For every internal $A\subseteq X$, the measure of 
$A$ is defined to be $\mu (A):=\st(\frac{|A|}{|X|})$. This defines a
finitely additive measure $\mu $ on the algebra of internal subsets of $X$,
which canonically extends to a countably additive probability measure $\mu
_{L}$ on the $\sigma $-algebra of \emph{Loeb measurable} sets of $X$.

\subsection{Amenable Groups}

Suppose that $G$ is a group. A \emph{(left) F\o lner sequence} for $G$ is a
sequence $(F_{n})_{n\in \mathbb{N}}$ of of finite subsets of $G$ such that,
for every $g\in G$, we have 
\begin{equation*}
\lim_{n\rightarrow \infty }\frac{|gF_{n}\triangle F_{n}|}{|F_{n}|}=0.
\end{equation*}%
Observe that if $(F_{n})$ is a F\o lner sequence for $G$ and $(x_{n})$ is
any sequence in $G$, then $(F_{n}x_{n})$ is also a F\o lner sequence for $G$%
. Observe also that, if $\nu \in {}^{\ast }\mathbb{N}\setminus \mathbb{N}$,
then $\frac{|gF_{\nu }\triangle F_{\nu }|}{|F_{\nu }|}\approx 0$ for every $%
g\in G$. (In the terminology of \cite{mauromartino}, $F_{\nu }$ is a \emph{F%
\o lner approximation for $G$}.)

A countable group $G$ is said to be \emph{amenable} if there is a F\o lner
sequence for $G$. For example, if $G=\mathbb{Z}$, then $G$ is amenable,
where one can take as $(F_{n})$ any sequence of intervals whose length
approaches infinity. The class of amenable groups is very rich, including
all solvable-by-finite groups, and is closed under subgroups, quotients, and
extensions.

In an amenable group, one can define a notion of (upper) Banach density. In
the rest of this subsection, fix a countable amenable group $G$. For $%
A\subseteq G$, the \emph{Banach density of $A$}, denoted $\BD(A)$, is
defined to be 
\begin{equation*}
\BD(A):=\sup \{\limsup_{n\rightarrow \infty }\frac{|A\cap F_{n}|}{%
|F_{n}|}\ :\ (F_{n})\text{ a F\o lner sequence for }G\}.
\end{equation*}

It can be shown that this supremum is actually attained in the sense that, for any $A\subseteq G$, there is a F\o lner sequence $(F_n)$ for $G$ such that $\lim_{n\to \infty}\frac{|A\cap F_n|}{|F_n|}=\BD(A)$. 

It is evident from
the definition that $\BD(A)=\BD\left( gA\right) =\BD(Ag)$ for all 
$g\in G$ and $A\subseteq G$. However, it is not a priori immediate that this
agrees with the usual notion of Banach density in the case that $G=\mathbb{Z}
$ as here one allows arbitrary F\o lner sequences rather than just sequences
of intervals. Nevertheless, it is shown in \cite[Remark 1.1]{BBF} that if $G$
is a countable amenable group and $(F_{n})$ is \emph{any} F\o lner sequence
for $G$, then there is a sequence $(g_{n})$ from $G$ such that $\BD%
(A)=\limsup_{n\rightarrow \infty }\frac{|A\cap F_{n}g_{n}|}{|F_{n}|}$,
whence we see immediately that the two notions of Banach density agree in
the case of the integers.

For finite $H\subseteq G$ and $\epsilon>0$, we say that a finite set $%
F\subseteq G$ is \emph{$(H,\epsilon)$-invariant} if, for every $h\in H$, we
have 
\begin{equation*}
\frac{|hF\triangle F|}{|F|}<\epsilon.
\end{equation*}
One can equivalently define a countable group to be amenable if, for every
finite $H\subseteq G$ and $\epsilon>0$, there is a finite subset of $G$ that
is $(H,\epsilon)$-invariant. (This definition has the advantage that it
extends to groups of arbitrary cardinality.) In this language, we have that $%
\BD(A)$ is the supremum of those $\gamma$ for which, given any finite $%
H\subseteq G$ and any $\epsilon>0$, there is a finite $F\subseteq G$ that is 
$(H,\epsilon)$-invariant and satisfying $\frac{|A\cap F|}{|F|}\geq \gamma$.


Finally, we will need a version of the pointwise ergodic theorem for
countable amenable groups due to E. Lindenstrauss \cite{Lind}. First, we say
that a F\o lner sequence $(F_{n})$ is \emph{tempered} if there is a constant 
$C>0$ such that, for every $n\in \mathbb{N}$, we have 
\begin{equation*}
|\bigcup_{k<n}F_{k}^{-1}F_{n}|\leq C|F_{n}|.
\end{equation*}%

For example, if $G=\mathbb Z$ and our $F_n$ are simply disjoint intervals with length and endpoints going to infinity, a tempered subsequence can always be obtained by insisting that the length of the $n^{\text{th}}$ interval in the subsequence is at least as large as the right endpoint of the $(n-1)^{\text{st}}$ interval.

Fortunately, there is an abundance of tempered F\o lner sequences for any
countable abelian group.

\begin{fact}[Lindenstrauss \protect\cite{Lind}]
Suppose that $G$ is a countable amenable group. Then every F\o lner sequence
for $G$ has a tempered subsequence. In particular, for $A\subseteq G$, there
is a tempered F\o lner sequence $(F_{n})$ for $G$ such that $\BD%
(A)=\lim_{n\rightarrow \infty }\frac{|A\cap F_{n}|}{|F_{n}|}$.
\end{fact}

Here is the pointwise ergodic theorem for countable amenable groups:

\begin{fact}[Lindenstrauss \protect\cite{Lind}]
\label{ergodic} Suppose that $G$ is a countable amenable group acting on a
probability space $(X,\mathcal{B},\mu )$ by measure preserving
transformations and $(F_{n})$ is a tempered F\o lner sequence for $G$. If $%
f\in L^{1}(\mu )$ and%
\begin{equation*}
A\left( F_{n},f\right) (x):=\frac{1}{|F_{n}|}\sum_{g\in F_{n}}f(gx)
\end{equation*}%
for every $n\in 
\mathbb{N}
$, then the sequence 
\begin{equation*}
\left( A\left( F_{n},f\right) \right) _{n\in 
\mathbb{N}
}
\end{equation*}%
converges almost everywhere to a $G$-invariant $\bar{f}\in L^{1}(\mu )$.
Consequently, by the Lebesgue dominated convergence theorem, $A\left(
F_{n},f\right) $ converges to $\bar{f}$ in $L^{1}(\mu )$ and, in particular, 
\begin{equation*}
\int fd\mu =\int \bar{f}d\mu \text{.}
\end{equation*}
\end{fact}

\subsection{A result of Bergelson}

Throughout our paper, we will make use of the following result of Bergelson, which is Theorem 1.1 in \cite{Bergelson}:

\begin{fact}
\label{bergelson} Suppose that $(X,\mathcal{B},\mu )$ is a probability space
and $(A_{n})$ is a sequence of measurable sets for which there is $a\in \r^{>0}$ such that $\mu (A_{n})\geq a$
for each $n$. Then there is infinite $P\subseteq \mathbb{N}$ 
such that, for every finite $F\subseteq P$, we have $\mu
(\bigcap_{n\in F}A_{n})>0$.
\end{fact}

\section{The high density case}

\label{high}

The main result of this section is the following:

\begin{thm}
\label{highdensity} Suppose that $G$ is a countable amenable group and $A\subseteq G$ is such that $\BD%
(A)>\frac{1}{2}$. Then there are injective sequences $(b_n)_{n\in \mathbb{N}%
} $ and $(c_n)_{n\in \mathbb{N}}$ in $G$ such that:

\begin{itemize}
\item $c_n\in A$ for all $n\in \mathbb{N}$;

\item $b_ic_j\in A$ for $i\leq j$;

\item $c_ib_j\in A$ for $i< j$.
\end{itemize}
\end{thm}

In the first subsection, we prove the analogous fact for subsets of the natural numbers as in this case we can avoid using Fact \ref{ergodic} and instead resort to more elementary methods.  We prove the case of a general amenable group in the second subsection.

\subsection{The case of the integers}

The main goal of this subsection is the following theorem.

\begin{thm}\label{highint}
Suppose that $A\subseteq \mathbb{N}$ is such that $\BD(A)>\frac{1}{2}$%
. Then there are infinite $B,C\subseteq \mathbb{N}$ with $C\subseteq A$ such that $B+C\subseteq
A $.
\end{thm}

We first need a lemma.

\begin{lemma}\label{easier}
Suppose that $A\subseteq \n$ has $\BD(A)=\alpha>0$.  Suppose that $(I_n)$ is a sequence of intervals with $|I_n|\to \infty$ and for which $\lim_{n\to \infty}\frac{|A\cap I_n|}{|I_n|}=\alpha$.  Then there is $L\subseteq \n$ satisfying:
\begin{itemize}
\item $\limsup_{n\to \infty}\frac{|L\cap I_n|}{|I_n|}\geq \alpha$;
\item for every finite $F\subseteq L$, we have $A\cap \bigcap_{x\in F}(A-x)$ is infinite.
\end{itemize}
\end{lemma}

\begin{proof}
It suffices to find $L\subseteq \n$ and $x_0\in {}^{\ast}\!{A}\setminus A$ for which $\limsup_{n\to \infty}\frac{|L\cap I_n|}{|I_n|}\geq \alpha$ and $x_0+L\subseteq {}^{\ast}\!{A}$.  Indeed, if we can find such $L$ and $x_0$, then given any finite $F\subseteq L$ and any finite $K\subseteq A$, the statement ``there exists $x_0\in {}^{\ast}\!{A}$ such that $x_0+F\subseteq {}^{\ast}\!{A}$ and $x_0\notin K$'' is true in the nonstandard extension, whence we can conclude that $A\cap \bigcap_{x\in F}(A-x)$ is infinite.

For each $n$, let $b_n$ denote the right endpoint of $I_n$.  By passing to a subsequence of $(I_n)$ if necessary, we may assume that the sequences $(b_n)$ and $(|I_n|)$ are strictly increasing.  Fix $H\in {}^{\ast}\n\setminus \n$ and note that $\frac{|{}^{\ast}\!{A}\cap I_H|}{|I_H|}\approx \alpha$.  

In what follows, we let $\mu$ denote the Loeb measure on $I_H$.  Also, for any $m\in {}^{\ast}\n$ (standard or nonstandard) and for any hyperfinite $X\subseteq {}^{\ast}\n$, we set $\delta_m(X):=\frac{|X|}{|I_m|}$.

We fix $K\in {}^{\ast}\n\setminus \n$ for which $2b_K\cdot \delta_H(I_K)\approx 0$ and consider $M\in {}^{\ast}\n\setminus \n$ with $M\leq K$.  We claim that, for $\mu$-almost all $x\in I_H$, we have $\delta_M({}^{\ast}\!{A}\cap (x+I_M))\approx \alpha$.  Indeed, since $\BD(A)=\alpha$, we can conclude that, for all $x\in I_H$, we have $\st(\delta_M({}^{\ast}\!{A}\cap (x+I_M)))\leq \alpha$.  We now compute
$$\frac{1}{|I_H|}\sum_{x\in I_H}\delta_M({}^{\ast}\!{A}\cap (x+I_M))=\frac{1}{|I_M|}\sum_{y\in I_M}\frac{1}{|I_H|}\sum_{x\in I_H}\chi_{{}^{\ast}\!{A}}(x+y).$$  
By the choice of $K$, it follows that $$\frac{1}{|I_H|}\sum_{x\in I_H}\delta_M({}^{\ast}\!{A}\cap (x+I_M))\approx \frac{1}{|I_M|}\sum_{y\in I_M}\delta_H({}^{\ast}\!{A}\cap I_H)\approx \alpha.$$  Coupled with our earlier observation, this proves the claim.

We now fix a standard positive real number $\epsilon<\frac{1}{2}$.  Inductively assume that we have chosen natural numbers $n_1<n_2<\cdots<n_{i-1}$ and internal subsets $X_1,X_2,\ldots,X_{i-1}\subseteq I_H$ such that, for each $j=1,2,\ldots,i-1$ and each $x\in X_j$, we have
$$\mu(X_j)>1-\epsilon^j \text{ and }\delta_{n_j}({}^{\ast}\!{A}\cap (x+I_{n_j}))\geq \alpha-\frac{1}{j}.$$

Consider the internal set 
$$Z:=\{M\in {}^{\ast}\n \ : \ n_{i-1}<M\leq K \text { and }$$ $$\delta_H\left(\left\{x\in I_H \ : \ \delta_M({}^{\ast}\!{A}\cap (x+I_M))\geq \alpha-\frac{1}{i}\right\}\right)>1-\epsilon^i\}.$$  Since $Z$ is internal and contains every nonstandard element of ${}^{\ast}\n$ below $K$, it follows that there is $n_i\in Z\cap \n$.  For this $n_i$, we set $$X_i:=\{x\in I_H \ : \ \delta_{n_i}({}^{\ast}\!{A}\cap (x+I_{n_i}))\geq \alpha-\frac{1}{i}\}.$$  Set $X:=\bigcap_{i=1}^\infty X_i$ and observe that $\mu(X)>0$.  Fix $y_0\in X$ and observe that, for all $i\in \n$, we have $$\delta_{n_i}({}^{\ast}\!{A}\cap (y_0+I_{n_i}))>\alpha-\frac{1}{i}.$$  

Set $x_0$ to be the minimum element of ${}^{\ast}\!{A} \cap [y_0,b_H]$ and set $$L:=({}^{\ast}\!{A}\cap (x_0+\n))-x_0.$$  Note that $x_0-y_0\in \n$ and $x_0+L\subseteq {}^{\ast}\!{A}$.  Since $x_0-y_0$ is finite, it follows that
$$\lim_{i\to \infty}\delta_{n_i}(L\cap I_{n_i})=\lim_{i\to \infty}\delta_{n_i}({}^{\ast}\!{A}\cap (x_0+I_{n_i}))=\lim_{i\to \infty}\delta_{n_i}({}^{\ast}\!{A}\cap (y_0+I_{n_i}))=\alpha.$$

\end{proof}

\begin{proof}[Proof of Theorem \ref{highint}]  Fix a sequence $(I_{n})$ of intervals such
that $|I_n|\to \infty$ and 
\begin{equation*}
\lim_{n\rightarrow \infty }\frac{|A\cap I_{n}|}{|I_{n}|}=\alpha .
\end{equation*}%
Fix $L$ as in the conclusion of Lemma \ref{easier}.  Let $L=(l_{n})$ be an increasing enumeration of $L$. Recursively define an increasing sequence $%
D:=(d_{n})_{n\in \mathbb{N}}$ from $A$ such that $l_{i}+d_{n}\in A$ for $i\leq n$%
. Fix $\nu \in {}^{\ast }\mathbb{N}\setminus \mathbb{N}$ such that $\st(\frac{|{}^{\ast }L\cap I_{\nu }|}{|I_{\nu }|})\geq \alpha $.  Recalling that $\alpha>\frac{1}{2}$, it follows that, for every $n\in \mathbb{N}$, we have 
\begin{equation*}
\st\left( \frac{|{}^{\ast }L\cap ({}^{\ast }A-d_n)\cap I_{\nu }|}{%
|I_{\nu }|}\right) \geq 2\alpha -1>0.
\end{equation*}%
By Fact \ref{bergelson}, we may, after passing to a subsequence of $(d_{n})$%
, assume that, for every $n\in \mathbb{N}$, we have 
\begin{equation*}
\st\left( \frac{|{}^{\ast }L\cap \bigcap_{i\leq n}({}^{\ast
}A-d_i)\cap I_{\nu }|}{|I_{\nu }|}\right) >0.
\end{equation*}%
In particular, this implies that, for every $n\in \mathbb{N}$, we have $%
L\cap \bigcap_{i\leq n}(A-d_i)$ is infinite. Take $b_{1}\in L$ arbitrary
and take $c_{1}\in D$ such that $b_{1}+c_{1}\in A$. Fix $b_{2}\in (L\cap
(A-c_1))\setminus \{b_{1}\}$ and take $c_{2}\in D$ such that $%
\{b_{1}+c_{2},b_{2}+c_{2}\}\subseteq A$. Take $b_{3}\in (L\cap (A-c_1)\cap
(A-c_2))\setminus \{b_{1},b_{2}\}$ and take $c_{3}\in D$ such that $%
\{b_{1}+c_{3},b_{2}+c_{3},b_{3}+c_{3}\}\subseteq A$. Continue in this way to
construct the desired $B$ and $C$.
\end{proof}

\subsection{The case of an arbitrary countable amenable group}

In this section, we assume that $G$ is a countable amenable group and prove Theorem \ref{highdensity}.  

Before proving Theorem \ref{highdensity}, we need a lemma analogous to Lemma \ref{easier}.

\begin{lemma}
\label{embed} Suppose that $(F_{n})$ is a tempered F\o lner sequence. If $%
A\subseteq G$ is such that $\limsup_{n\rightarrow \infty }\frac{|A\cap F_{n}|%
}{|F_{n}|}=\alpha $, then there is $L\subseteq G$ satifying:

\begin{itemize}
\item $\liminf_{n\to \infty} \frac{|L\cap F_n|}{|F_n|}\geq \alpha$;

\item for every finite $F\subseteq L$, we have $A\cap \bigcap_{x\in
F}x^{-1}A $ is infinite.
\end{itemize}
\end{lemma}

\begin{proof}
Fix $\nu \in {}^{\ast }\mathbb{N}\setminus \mathbb{N}$ such that $\frac{%
|{}^{\ast }\!{A}\cap F_{\nu }|}{|F_{\nu }|}\approx \alpha $. Notice that, for
all $g\in G$, we have $\frac{|gF_{\nu }\triangle F_{\nu }|}{|F_{\nu }|}%
\approx 0$. Since $G$ is countable, there is a full measure (with respect to
the Loeb measure on $F_{\nu }$) subset $E$ of $F_{\nu }$ for which the map $%
(g,x)\mapsto gx:G\times E\rightarrow E$ defines a measure preserving action
of $G$ on $E$. For $\xi \in E$, we define 
\begin{equation*}
f_{n}(\xi ):=A\left( F_{n},\chi _{{}^{\ast }\!{A}\cap E}\right) \left( \xi
\right) =\frac{1}{|F_{n}|}\sum_{g\in F_{n}}\chi _{{}^{\ast }\!{A}\cap E}(g\xi )
\end{equation*}%
where $\chi _{{}^{\ast }\!{A}\cap E}$ denotes the characteristic function of $%
{}^{\ast }\!{A}\cap E$. Observe that 
\begin{equation*}
f_{n}(\xi )=\frac{|F_{n}\cap ({}^{\ast }\!{A}\cap E)\xi ^{-1}|}{|F_{n}|}\leq 
\frac{|F_{n}\cap {}^{\ast }\!{A}\xi ^{-1}|}{|F_{n}|}.\quad (\dagger )
\end{equation*}%
By Fact \ref{ergodic}, there is $\bar{f}\in L^{1}(\mu )$ such that $(f_{n})$
converges to $\bar{f}$ almost everywhere and in $L^{1}(\mu )$, whence $\int 
\bar{f}d\mu =\alpha $. (Here, $\mu $ denotes the restriction of the Loeb
measure on $F_{\nu }$ to $E$.)

We next claim that $\bar{f}$ is almost everywhere bounded above by $\alpha $%
. If this is not the case, then there is $k\in \mathbb{N}$ such that the set
of $\xi \in E$ for which $\bar{f}(\xi )\geq \alpha +\frac{1}{k}$ has
positive measure. Since $f_{n}$ converges to $\bar{f}$ almost everywhere,
there is $\xi \in E$ such that $\lim_{n\rightarrow \infty }f_{n}(\xi )\geq
\alpha +\frac{1}{k}$, whence, by $(\dagger )$, we have 
\begin{equation*}
\liminf_{n\rightarrow \infty }\frac{|F_{n}\xi \cap {}^{\ast }\!{A}|}{|F_{n}|}%
\geq \alpha +\frac{1}{k}.
\end{equation*}%
By transfer, for each $n\in 
\mathbb{N}
$ there is $x_{n}\in G$ such that%
\begin{equation*}
|F_nx_n\cap A|=|F_n\xi\cap {}^{\ast }\!{A}|.
\end{equation*}%
Since $\left( F_{n}x_{n}\right) $ is also a F\o lner sequence for $G$ this
implies%
\begin{equation*}
\BD\left( A\right) \geq \limsup_{n}\frac{\left\vert F_{n}x_{n}\cap
A\right\vert }{\left\vert F_{n}\right\vert }\geq \alpha +\frac{1}{k}.
\end{equation*}%
This contradicts the fact that $\BD(A)=\alpha $.

By our claim and the fact that $\int \bar{f}d\mu =\alpha $, we see that $%
\bar{f}$ is almost everywhere equal to $\alpha $. In particular, there is $%
\xi \in {}^{\ast }\!{A}\cap E$ such that $\lim_{n\rightarrow \infty }f_{n}(\xi
)=\alpha $. Since $G\cap E$ has measure $0$, whence we can
further insist that $\xi \in {}^{\ast }\!{A}\setminus G$. Fix such $\xi $ and
set $L:={}^{\ast }\!{A}\xi ^{-1}\cap G$. By $(\dagger )$ and the choice of $\xi $%
, we have $\liminf_{n\rightarrow \infty }\frac{|L\cap F_{n}|}{|F_{n}|}\geq
\alpha $.

It remains to show that $A\cap \bigcap_{x\in F}x^{-1}A$ is infinite for
every finite subset $F$ of $L$. Fix such an $F$. For each $x\in F$, we have $%
x\xi \in {}^{\ast }\!{A}$. Since $\xi \notin G$, for any finite $K\subseteq G$,
the statement \textquotedblleft there exists $h\in {}^{\ast }\!{A}$ such that $%
h\notin K$ and, for every $x\in F$, we have $xh\in {}^{\ast }\!{A}$%
\textquotedblright\ holds in the nonstandard extension. Thus, by transfer,
for any given finite subset $K$ of $G$, there is $h\in A$ such that $h\notin
K$ and $xh\in A$ for each $x\in F$.
\end{proof}

The proof of Theorem \ref{highdensity} from Lemma \ref{embed} is almost the same as the proof of Theorem \ref{highint} from Lemma \ref{easier}, but we include the proof for the sake of the reader.

\begin{proof}[Proof of Theorem \ref{highdensity}] Fix $A\subseteq G$ such that $\alpha :=\BD(A)>\frac{1}{2}$. Fix a tempered F\o lner sequence $(F_{n})$ for $G$ such
that 
\begin{equation*}
\lim_{n\rightarrow \infty }\frac{|A\cap F_{n}|}{|F_{n}|}=\alpha .
\end{equation*}%
Fix $L$ as in the conclusion of Lemma \ref{embed}. Fix an injective
enumeration $L=(l_{n})$ of $L$. Recursively define an injective sequence $%
D=(d_{n})_{n\in \mathbb{N}}$ from $A$ such that $l_{i}d_{n}\in A$ for $i\leq n$%
. Fix $\nu \in {}^{\ast }\mathbb{N}\setminus \mathbb{N}$. For any $g\in G$,
we have 
\begin{equation*}
\st\left( \frac{|g{}^{\ast }\!{A}\cap F_{\nu }|}{|F_{\nu }|}\right) =\st \left( \frac{|g{}^{\ast }\!{A}\cap gF_{\nu }|}{|F_{\nu }|}\right) =\st%
\left( \frac{|{}^{\ast }\!{A}\cap F_{\nu }|}{|F_{\nu }|}\right) =\alpha \text{;}
\end{equation*}%
since we also have $\st(\frac{|{}^{\ast }L\cap F_{\nu }|}{|F_{\nu }|}%
)\geq \alpha $, it follows that, for every $n\in \mathbb{N}$, we have 
\begin{equation*}
\st\left( \frac{|{}^{\ast }L\cap d_{n}^{-1}{}^{\ast }\!{A}\cap F_{\nu }|}{%
|F_{\nu }|}\right) \geq 2\alpha -1>0.
\end{equation*}%
By Fact \ref{bergelson}, we may, after passing to a subsequence of $(d_{n})$%
, assume that, for every $n\in \mathbb{N}$, we have 
\begin{equation*}
\st\left( \frac{|{}^{\ast }L\cap \bigcap_{i\leq n}d_{i}^{-1}{}^{\ast
}A\cap F_{\nu }|}{|F_{\nu }|}\right) >0.
\end{equation*}%
In particular, this implies that, for every $n\in \mathbb{N}$, we have $%
L\cap \bigcap_{i\leq n}d_{i}^{-1}A$ is infinite. Take $b_{1}\in L$ arbitrary
and take $c_{1}\in D$ such that $b_{1}c_{1}\in A$. Fix $b_{2}\in (L\cap
c_{1}^{-1}A)\setminus \{b_{1}\}$ and take $c_{2}\in D$ such that $%
\{b_{1}c_{2},b_{2}c_{2}\}\subseteq A$. Take $b_{3}\in (L\cap c_{1}^{-1}A\cap
c_{2}^{-1}A)\setminus \{b_{1},b_{2}\}$ and take $c_{3}\in D$ such that $%
\{b_{1}c_{3},b_{2}c_{3},b_{3}c_{3}\}\subseteq A$. Continue in this way to
construct the desired $B$ and $C$.
\end{proof}

We say that $(F_{n})$ is a \emph{two-sided F\o lner sequence} for $G$ if,
for all $g\in G$, we have 
\begin{equation*}
\lim_{n\rightarrow \infty }\frac{|(gF_{n}\triangle F_{n})|+|(F_{n}g\triangle
F_{n})|}{|F_{n}|}=0.
\end{equation*}%
Of course, if $G$ is abelian, then every F\o lner sequence is two-sided. If $%
G$ is amenable, then two-sided F\o lner sequences for $G$ exist. However, it is unclear, given $A\subseteq G$ with positive Banach
density,  whether or not there is a two-sided F\o lner sequence for $G$ witnessing the
Banach density of $A$. 

If we repeat the previous proof with $Ad_n^{-1}$ instead of $d_n^{-1}A$, we
get the following result.

\begin{thm}
Suppose that $(F_{n})$ is a two-sided F\o lner sequence for $G$ and $%
A\subseteq G$ is such that $\lim_{n\rightarrow \infty }\frac{|A\cap F_{n}|}{%
|F_{n}|}=\BD(A)>\frac{1}{2}$. Then there are infinite $B,C\subseteq G$
with $C\subseteq A$ such that $BC\subseteq A$.
\end{thm}

Let us end this section by showing how to derive Theorem \ref{highint} from Theorem \ref{highdensity} directly.  Suppose that $A\subseteq \mathbb{N}$ has Banach density exceeding $\frac{1}{2}
$. Then $\BD(A)>\frac{1}{2}$ \emph{when viewed as a subset of }$%
\mathbb{Z}$. By Theorem \ref{highdensity}, there are infinite sequences $%
B,C\subseteq \mathbb{Z}$ such that $C\subseteq A$ and $B+C\subseteq A$.
Since $C\subseteq A\subseteq \mathbb{N}$, this forces all but finitely many
elements of $B$ to belong to $\mathbb{N}$; replacing $B$ with $B\cap \mathbb{%
N}$ yields the desired result.

%
%

\section{A one-shift result for sets of positive Banach density}

\label{general}

The main result of this section is the following.

\begin{thm}
\label{smalldensity} If $A\subseteq G$ has positive Banach density, then
there are injective sequences $\left( b_{n}\right) _{n\in 
\mathbb{N}
}$ and $\left( c_{n}\right) _{n\in 
\mathbb{N}
}$ in $G$ and $h,h^{\prime }\in G$ such that:

\begin{itemize}
\item $c_{n}\in A$ for each $n$;

\item $b_{i}c_{j}\in hA$ for $i\leq j$;

\item $c_{j}b_{i}\in h^{\prime }A$ for $i< j$.
\end{itemize}
\end{thm}

The proof proceeds in two steps.  First, we show that we can ``fatten'' $A$ to a set $QA$, where $Q\subseteq G$ is finite, for which $\BD(QA)>\frac{1}{2}$.  We then apply Theorem \ref{highdensity} to $QA$ and apply Ramsey's theorem to obtain the desired result.  The first step was done in \cite{Hindman2} in the case of the natural numbers, so we cover this case separately for those readers who are primarily interested in the case of subsets of the natural numbers.   

\subsection{The case of the integers}

\begin{df}
For $A\subseteq \n$ and $n\in\mathbb{N}$, we define $A_{[n]}\subseteq \n$
by declaring $k\in A_{[n]}$ iff $[kn,kn+n-1]\cap A\neq\varnothing$. \ In other words, if the natural
numbers are partitioned into equal sized blocks of length $n$, then $A_{[n]}$
is the sequence of the ``block numbers'' that intersect $A$. 
\end{df}

The following is Theorem 3.8 in \cite{Hindman2}.

\begin{fact}\label{fatintegers}
For any $A$ with $\BD(A)>0$ and any $\epsilon>0$, there exists $n\in \n$ such that
$\BD(A_{[n]})\geq1-\epsilon$.
\end{fact}

%

We are now ready to prove the one-shift result in the case of subsets of the natural numbers.

\begin{thm}\label{twotranslate}
If $A\subseteq \n$ is such that $\BD(A)>0$, then there exist infinite
sets $B,C\subseteq \n$ and $k\in \z$ such that $B+C$ $\subset A\cup(A+k)$.
\end{thm}

\begin{proof}
By the previous lemma, there exists $n\in \n$ such that $\BD(A_{[n]})>\frac{1}{2}$.  Applying Theorem \ref{highdensity} to $A_{[n]}$, we obtain sets $B_{[n]}=(b_{i}),C_{[n]}=(c_{j})$ such that
$B_{[n]}+C_{[n]}\subset A_{[n]}$. \ In other words, every $[nb_{i}+nc_{j},nb_{i}%
+nc_{j}+n-1]$ intersects $A$. \ Using $n^{2}$ colors we may code every pair of
natural numbers $\{i,j\}$ with $i<j$ based on which $\nu\in\lbrack0,n-1]$ is
such that $nb_{i}+nc_{j}+\nu$ is the first element of $A$ in $[nb_{i}%
+nc_{j},nb_{i}+nc_{j}+n-1]$, and which $\xi\in\lbrack0,n-1]$ is such that
$nc_{i}+nb_{j}+\xi$ is the first element of $A$ in $[nc_{i}+nb_{j}%
,nc_{i}+nb_{j}+n-1]$. \ By Ramsey's theorem, there exists an infinite $J\subseteq \n$ monochromatic for this coloring.  

We now replace $B_{[n]}$ and $C_{[n]}$ by infinite subsequences whose indices come from $J$.  In particular, there is a fixed pair $\nu$ and $\xi$ such
that, for any $i<j$, $nb_{i}+nc_{j}+\nu\in A$ while $nc_{i}+nb_{j}+\xi\in A.$
\ If we now let $k=\nu-\xi$, $B=\{nb_{i}+\nu:i$ is odd$\}$, and $C=\{nc_{j}:j$ is even$\}$, we see that $B+C\subset A\cup(A+k)$, with the translate of $A$ for a given element of $B+C$ determined
by whether $i<j$ or $i>j$. \ It is important to note that by taking only the odd indices from one set
and the even indices from the other set we avoid the case in which the indices
are the same, something that was not determined by the use of Ramsey's Theorem.
\end{proof}

\subsection{The case of an arbitrary amenable group}

In this subsection, we once again assume that $G$ is a countable amenable group.

In order to prove the analog of Fact \ref{fatintegers} in the case of an arbitrary amenable group, we
will need the following fact, which is a particular case of Theorem 4.5 in 
\cite{Pogorzelski-Schwarzenberger}.  (There one assumes that the amenable group is unimodular, which is immediate in our case since our groups are discrete.)

\begin{fact}
\label{Lemma: Pogo-Schwa}Suppose that $\varepsilon \in \left( 0,\frac{1}{10}%
\right) $. Define $N(\varepsilon )=\left\lceil \frac{\mathrm{log}\left(
\varepsilon \right) }{\mathrm{log}\left( 1-\varepsilon \right) }\right\rceil 
$. For every finite subset $H$ of $G$ and every $\delta \in \left(
0,\varepsilon \right) $ there are $\left( H,\delta \right) $-invariant
finite sets%
\begin{equation*}
\left\{ 1_{G}\right\} \subset T_{1}\subset T_{2}\subset \ldots \subset
T_{N(\varepsilon )}\text{,}
\end{equation*}%
a finite subset $K$ of $G$ containing $H$, and a positive real number $\eta
<\delta $ such that, for every finite subset $F$ of $G$ which is $\left(
K,\eta \right) $-invariant, there are finite sets $C_{i}\subset G$ and $%
T_{i}^{\left( c\right) }\subset T_{i}$ for $i=1,2,\ldots ,N\left(
\varepsilon \right) $ such that:

\begin{itemize}
\item $\left\{ T_{i}^{\left( c\right) }c\left\vert\  i\leq N(\varepsilon )%
\text{, }c\in C_{i}\right. \right\}$ is a family of pairwise
disjoint sets;

\item $\bigcup_{i=1}^{N(\varepsilon )}T_{i}^{\left( c\right) }C_{i}\subset F$%
;

\item $\left\vert \bigcup_{i=1}^{N(\varepsilon )}T_{i}^{\left( c\right)
}C_{i}\right\vert>\left( 1-2\varepsilon \right)\cdot |F| $.
\end{itemize}
\end{fact}

\begin{lemma}
\label{fat}
For any $A\subseteq G$ with $\BD(A)>0$ and for every $\rho >0$, there is a finite subset $Q$ of $G$ such
that $\BD\left( QA\right) >1-\rho $.
\end{lemma}

\begin{proof}

Set $\alpha:=\BD(A)>0$.  Pick $\varepsilon >0$ such that%
\begin{equation*}
\frac{\alpha -3\varepsilon }{\alpha +\varepsilon }>1
-\rho \text{.}
\end{equation*}%
Since $\alpha+\epsilon>\BD(A)$, there is finite $H\subseteq G$ and $\delta \in \left( 0,\varepsilon
\right) $ such that, for every $\left( H,\delta \right) $-invariant set $F
$, we have%
\begin{equation*}
\frac{\left\vert F\cap A\right\vert }{\left\vert F\right\vert }<\alpha
+\varepsilon \text{.}
\end{equation*}%
Fix $K\subseteq G$ finite, $\eta >0$, and%
\begin{equation*}
\left\{ 1_{G}\right\} \subset T_{1}\subset T_{2}\subset \ldots \subset
T_{N(\varepsilon )}
\end{equation*}%
obtained from $\varepsilon $, $\delta $, and $H$ as in Fact \ref{Lemma:
Pogo-Schwa}. Define%
\begin{equation*}
Q=\bigcup_{i=1}^{N(\varepsilon )}T_{i}T_{i}^{-1}
\end{equation*}%
and%
\begin{equation*}
B=QA\text{.}
\end{equation*}%
We claim that $\BD(B)>1-\rho$.  Towards this end, fix a F\o lner sequence $\left( F_{n}\right) _{n\in 
\mathbb{N}
}$ of $G$ such that%
\begin{equation*}
\limsup_{n}\frac{\left\vert A\cap F_{n}\right\vert }{\left\vert
F_{n}\right\vert }=\alpha \text{.}
\end{equation*}%
We claim that%
\begin{equation*}
\limsup_{n}\frac{\left\vert B\cap F_{n}\right\vert }{\left\vert
F_{n}\right\vert }> 1-\rho\text{%
.}
\end{equation*}%
Fix $n_{0}\in 
\mathbb{N}
$ and pick $n\geq n_{0}$ such that $F_{n}$ is $\left( K,\eta \right) $%
-invariant and%
\begin{equation*}
\frac{\left\vert F_{n}\cap A\right\vert }{\left\vert F_{n}\right\vert }%
>\alpha -\varepsilon \text{.} \quad (\dagger)
\end{equation*}%
Fix sets $C_{i}$ and $T_{i}^{\left( c\right) }\subset T_{i}$ for $i\leq
1,2,\ldots ,N(\varepsilon )$ obtained from $F_{n}$ as in Fact \ref{Lemma:
Pogo-Schwa}. Define%
\begin{equation*}
\mathcal{T}=\left\{ T_{i}^{\left( c\right) }c\left\vert \ i\leq N\left(
\varepsilon \right) ,c\in C_{i}\right. \right\} 
\end{equation*}%
and observe that $\mathcal{T}$ is a finite family of pairwise disjoint $%
\left( H,\delta \right) $-invariant finite sets such that%
\begin{equation*}
\frac{\left\vert \bigcup \mathcal{T}\right\vert }{\left\vert
F_{n}\right\vert }>1-2\varepsilon \text{.} \quad (\dagger \dagger)
\end{equation*}%
Define 
\begin{equation*}
\mathcal{T}_{0}=\left\{ T\in \mathcal{T}\left\vert \,T\cap A\neq \varnothing
\right. \right\} .
\end{equation*}%
We have%
\begin{eqnarray*}
\left( \alpha -\varepsilon \right) \left\vert F_{n}\right\vert  &<&\left\vert
A\cap F_{n}\right\vert  \\
&\leq &\left\vert A\cap \bigcup \mathcal{T}\right\vert +2\varepsilon
\left\vert F_{n}\right\vert  \\
&=&\left\vert A\cap \bigcup \mathcal{T}_{0}\right\vert +2\varepsilon
\left\vert F_{n}\right\vert  \\
&=&\sum_{T\in \mathcal{T}_{0}}\left\vert A\cap T\right\vert +2\varepsilon
\left\vert F_{n}\right\vert  \\
&\leq &\sum_{T\in \mathcal{T}_{0}}\left\vert T\right\vert \left( \alpha
+\varepsilon \right) +2\varepsilon \left\vert F_{n}\right\vert  \\
&=&\left\vert \bigcup \mathcal{T}_{0}\right\vert \left( \alpha +\varepsilon
\right) +2\varepsilon \left\vert F_{n}\right\vert .
\end{eqnarray*}%
In the above string of equalities and inequalities, the first line follows from $(\dagger)$, the second line follows from $(\dagger\dagger)$, the third line follows from the definition of $\mathcal{T}_0$, the fourth line follows from the fact that the members of $\mathcal{T}_0$ are pairwise disjoint, and the fifth line follows from the fact that the elements of $\mathcal{T}_0$ are $(H,\delta)$-invariant and the choice of $H$ and $\delta$.

It follows that %
\begin{equation*}
\frac{\left\vert \bigcup \mathcal{T}_{0}\right\vert }{\left\vert
F_{n}\right\vert }\geq \frac{\alpha -3\varepsilon }{\alpha +\varepsilon }%
\text{.}
\end{equation*}%
Observe that $B\supset \bigcup \mathcal{T}_{0}$ and therefore%
\begin{equation*}
\frac{\left\vert B\cap F_{n}\right\vert }{\left\vert F_{n}\right\vert }\geq 
\frac{\left\vert \bigcup \mathcal{T}_{0}\right\vert }{\left\vert
F_{n}\right\vert }\geq \frac{\alpha -3\varepsilon }{\alpha +\varepsilon }>%
1-\rho \text{.}
\end{equation*}
\end{proof}

Theorem \ref{smalldensity} now follows from Lemma \ref{fat} in the same way that Theorem \ref{twotranslate} followed from Lemma \ref{fatintegers}.

We leave it to the reader to verify that Theorem \ref{twotranslate} also follows from the special case of Theorem \ref{smalldensity} for $G=\mathbb \z$.

\section{The pseudorandom case}

\label{pseudo}

In this section, we prove that the $B+C$ conjecture holds for $A$ that are 
\emph{pseudorandom} in a sense to be described below. We start by recalling
some preliminary facts and definitions.

Suppose that $H$ is a Hilbert space and $U:H\to H$ is a unitary operator. We
say that $x\in H$ is \emph{weakly mixing (for $U$)} if 
\begin{equation*}
\lim_{n\to \infty} \frac{1}{n} \sum_{i=1}^n |\langle U^ix,x\rangle|=0.
\end{equation*}
We will need the following result; see \cite[Theorem 3.4]{Krengel} for a proof.

\begin{fact}
\label{mixing} $x\in H$ is weakly mixing if and only if $\lim_{n\to \infty} 
\frac{1}{n}\sum_{i=1}^n |\langle U^ix,y\rangle|=0$ for every $y\in H$.
\end{fact}

We will also need the following easy fact.

\begin{fact}
\label{cesarodensity} Suppose that $(r_n)$ is a sequence of nonnegative real
numbers. Then $\lim_{n\to \infty} \frac{1}{n}\sum_{i=1}^n r_n=0$ if and only
if, for every $\epsilon>0$, we have 
\begin{equation*}
\underline{d}(\{n\in \mathbb{N} \ : \ r_n\leq \epsilon\})=1.
\end{equation*}
\end{fact}

In what follows, we will need the notion of \emph{upper (asymptotic) density}%
. For $A\subseteq \mathbb{N}$, the upper density of $A$, denoted $\overline{d%
}\left( A\right) $, is defined to be 
\begin{equation*}
\overline{d}\left( A\right) :=\limsup_{n\rightarrow \infty }\frac{|A\cap
\lbrack 1,n]|}{n}.
\end{equation*}

For $N\in {}^{\ast}\mathbb{N}\setminus \mathbb{N}$, we set $%
A_N:={}^{\ast }\!{A}\cap [1,N]$ and write $\mu_N$ for the Loeb measure on $[1,N]$%
. We always consider $[1,N]$ to be equipped with its Loeb measure $\mu_N$.

Suppose that $A\subseteq \mathbb{N}$ is such that $\overline{d}(A)=\alpha >0$
and $N\in {}^{\ast }\mathbb{N}\setminus \mathbb{N}$ is such that $\frac{%
|A_{N}|}{N}\approx \alpha $. Notice that $\chi _{A_{N}}\in L^{2}(\mu _{N})$.
We have a measure preserving transformation $T:[1,N]\rightarrow \lbrack 1,N]$
defined by 
\begin{equation*}
T(x):=x+1\ (\mo N).
\end{equation*}%
The transformation $T$ gives rise to the unitary operator $U_{T}:L^{2}(\mu
_{N})\rightarrow L^{2}(\mu _{N})$ given by $U_{T}(f):=f\circ T$.

We are now ready to define our notion of pseudorandom.

\begin{df}
Suppose that $A\subseteq \mathbb{N}$ is such that $\overline{d}(A)=\alpha>0$%
. We say that $A$ is \emph{pseudorandom} if there is $N \in {}^{\ast}\mathbb{%
N}\setminus \mathbb{N}$ such that, in the notation preceding the definition,
we have that $\chi_{A_N}-\alpha$ is weakly mixing (for $U_T$).
\end{df}

Equivalently $A$ is pseudorandom if and only if there is $N\in {}^{\ast }%
\mathbb{N}\setminus \mathbb{N}$ as above such that 
\begin{equation*}
\lim_{n\rightarrow \infty }\frac{1}{n}\sum_{i=1}^{n}\left\vert \mu
_{N}({}A_{N}\cap (A-i)_{N})-\alpha ^{2}\right\vert =0\text{.}
\end{equation*}

It appears to be a little awkward to give a standard reformulation of the
aforementioned notion of pseudorandom. Certainly, if there is an increasing
sequence $(b_k)$ of natural numbers such that:

\begin{itemize}
\item $\lim_k b_k=\infty$,

\item $\lim_k \frac{|A\cap [1,b_k]|}{b_k}=\alpha$, and

\item $\lim_{n}\frac{1}{n}\sum_{i=1}^{n}\lim_{k}\left\vert \frac{|A\cap
(A-i)\cap \lbrack 1,b_{k}]|}{b_{k}}-\alpha ^{2}\right\vert =0$,
\end{itemize}

\noindent then $A$ is pseudorandom (just take $N=b_K$ for any $K\in {}^{\ast}%
\mathbb{N}\setminus \mathbb{N}$).

In order to prove that pseudorandom sets satisfy the $B+C$ conjecture, we
will need one last fact whose proof is nearly identical to the proof of Theorem 4.6 in \cite{DiNasso} (just replace arbitrary hyperfinite intervals by hyperfinite intervals of the form $[1,N]$).

\begin{fact}
\label{mauro} If $A\subseteq \mathbb{N}$ is such that $\overline{d}(A)>0$,
then there is $L\subseteq \mathbb{N}$ such that $\underline{d}(L)=\overline{d%
}(A)$ and 
\begin{equation*}
\overline{d}\left( \bigcap_{l\in F} (A-l)\right)>0
\end{equation*}
for every finite $F\subseteq L$.
\end{fact}

We are now ready to prove the main result of this section.

\begin{thm}\label{pseudoBC}
\label{defpseudo} If $A\subseteq \mathbb{N}$ is pseudorandom, then there are
infinite $B,C\subseteq \mathbb{N}$ such that $B+C\subseteq A$.
\end{thm}

\begin{proof}
Set $\alpha :=\overline{d}(A)$ and take $N$ as above witnessing that $A$ is
pseudorandom. For ease of notation, we set $\mu :=\mu _{N}$. By Fact \ref%
{mauro}, we may fix $L=(l_{n})$ with $\underline{d}(L)=\alpha $ and such
that 
\begin{equation*}
\overline{d}\left( \bigcap_{l\in F}(A-l)\right) >0
\end{equation*}%
for every finite $F\subseteq L$. Set $\beta :=\mu (L_{N})\geq \alpha $.
Observe that $U_{T}^{i}(\chi _{A_{N}})=\chi _{(A-i)_{N}}$. Since $\chi
_{A_{N}}-\alpha $ is weak mixing, by Fact \ref{mixing}, we have 
\begin{equation*}
\lim_{n\rightarrow \infty }\frac{1}{n}\sum_{i=1}^{n}|\mu ((A-i)_{N}\cap
L_{N})-\alpha \beta |=0.
\end{equation*}%
By Fact \ref{cesarodensity}, for every $\epsilon >0$, we have that 
\begin{equation*}
R_{\epsilon }:=\{n\in \mathbb{N}\ :\ |\mu ((A-n)_{N}\cap L_{N})-\alpha \beta
|<\epsilon \}
\end{equation*}%
has lower density equal to $1$. In particular, for any $\epsilon >0$ and any
finite $F\subseteq L$, we have that 
\begin{equation*}
\overline{d}\left( \bigcap_{l\in F}(A-l)\cap R_{\epsilon }\right) >0.
\end{equation*}%
Setting $\eta :=\frac{\alpha ^{2}}{2}$, this allows us to inductively define
a sequence $(d_{n})$ such that, for each $n\in \mathbb{N}$, we have $%
d_{n}\in \bigcap_{i\leq n}(A-l_{i})\cap R_{\eta }$. In particular, we have $%
\mu ((A-d_{n})_{N}\cap L_{N})>\eta $ for each $n\in \mathbb{N}$. We now
apply Fact \ref{bergelson} to the family $((A-d_{n})_{N}\cap L_{N})$ to get
a subsequence $(e_{n})$ of $(d_{n})$ such that 
\begin{equation*}
\mu (\bigcap_{i\leq n}(A-e_{i})_{N}\cap L_{N})>0
\end{equation*}%
for each $n\in \mathbb{N}$. Finally, as in the proof of Theorem \ref%
{highdensity}, this allows us to define subsequences $B=(b_{n})$ and $%
C=(c_{n})$ of $(l_{n})$ and $(e_{n})$, respectively, for which $B+C\subseteq
A$.
\end{proof}

We end this section with a question. First, for $H$ a Hilbert space and $%
U:H\to H$ a unitary operator, we say that $x\in H$ is \emph{almost periodic
(for $U$)} if $\{U^nx \ : \ n\in \mathbb{Z}\}$ is relatively compact (in the
norm topology). Using the notation of Definition \ref{defpseudo}, we say
that $A$ is \emph{almost periodic} if $\chi_{A_N}$ is an almost periodic
element of $L^2([0,N])$ (for $U_T$).

\begin{question}
If $A$ is almost periodic, does $A$ satisfy the conclusion of the $B+C$
conjecture?
\end{question}

This distinction between weakly mixing and almost periodic subsets of $\n$ is reminiscent of
Furstenberg's proof of Szemeredi's Theorem (see \cite{Furstenberg}), where
it is shown how to prove Szemeredi's Theorem by first establishing it for
the weakly mixing and compact cases and then showing how to derive it for
the general case by ``Furstenberg towers'' that are ``built from'' both of
these cases. It thus makes sense to ask:

\begin{question}\label{decomp}
If the previous question has an affirmative answer, is there a way to
decompose an arbitrary $A\subseteq \mathbb{N}$ of positive lower density
into a ``tower'' built from weakly mixing and almost periodic parts in a way
that allows one to prove the $B+C$ conjecture?
\end{question}

It is unclear to us whether there are many concrete examples of pseudorandom subsets of the natural numbers, but we believe the value of Theorem \ref{pseudoBC} is that it may be a first step in proving the $B+C$ conjecture via the route outlined in Question \ref{decomp}.

\end{document}